\documentclass[12pt,a4paper]{amsart}
\usepackage{amsfonts}
\usepackage{amsthm}
\usepackage{amsmath}
\usepackage{amssymb}
\usepackage{amscd}
\usepackage[latin2]{inputenc}
\usepackage{t1enc}
\usepackage[mathscr]{eucal}
\usepackage{indentfirst}
\usepackage{graphicx}
\usepackage{graphics}
\usepackage{pict2e}
\usepackage{epic}
\numberwithin{equation}{section}
\usepackage[margin=2.9cm]{geometry}
\usepackage{epstopdf} 
\usepackage{url}

\theoremstyle{plain}
\newtheorem{theorem}{Theorem}[section]
\newtheorem{lemma}[theorem]{Lemma}
\newtheorem{corollary}[theorem]{Corollary}
\newtheorem{proposition}[theorem]{Proposition}

\theoremstyle{definition}
\newtheorem{definition}[theorem]{Definition}

\newtheorem{remark}[theorem]{Remark}
\newtheorem{?}[theorem]{Problem}

\newcommand{\ovl}{\overline}
\newcommand{\udb}{\underbrace}

\newcommand{\imp}{\Rightarrow}

\newcommand{\ints}{\mathbb{Z}}
\newcommand{\nat}{\mathbb{N}}

\newcommand{\union}{\cup}
\newcommand{\intersect}{\cap}
\newcommand{\inv}{^{-1}}

\newcommand{\Y}{(X^\$ \times X^\$)^*}
\newcommand{\Z}{({X^\$})^*}
\newcommand{\X}{{X^\$}}

\newcommand{\vphi}{\varphi}
\renewcommand{\L}{\mathcal{L}}

\newcommand{\padL}{L^\$}

\begin{document}

\title[Language convexity]{Formal language convexity in left-orderable groups}

\author[HL Su]{Hang Lu Su}

\address{Dpto. de Matem\'aticas, Universidad Aut\'onoma de Madrid and Instituto de
Ciencias Matem\'aticas, CSIC-UAM-UC3M-UCM.}
\email{homeowmorphism@gmail.com}
\date{}


\thanks{
The author would like to thank Yago Antol\'in for his exceptionally thoughtful supervision, and for his guidance on the problems tackled in this paper. The author would also like to thank the anonymous referee for their detailed feedback and their suggestion of a different choice of transversal for Proposition \ref{H-pres}, which simplified the rest of section 5. The author has received funding from ``la Caixa" Foundation (ID 100010434) with fellowship code LCF/BQ/IN17/11620066, from the European Union's Horizon 2020 research and innovation programme under the Marie Sk\l{}odowska-Curie grant agreement No. 713673, and from the Severo Ochoa Programme for Centres of Excellence in R\&D (SEV-20150554).}

\begin{abstract} We propose a criterion for preserving the regularity of a formal language representation when passing from groups to subgroups.  We use this criterion to show that the regularity of a positive cone language in a left-orderable group passes to its finite index subgroups, and to show that there exists no left order on a finitely generated acylindrically hyperbolic group such that the corresponding positive cone is represented by a quasi-geodesic regular language. We also answer one of Navas' questions by giving an example of an infinite family of groups which admit a positive cone that is generated by exactly $k$ generators, for every $k \geq 3$. As a special case of our construction, we obtain a finitely generated positive cone for $F_2 \times \ints$. 
\end{abstract}

\maketitle

\section{Introduction}
A language represents a subset of a group if its image under the evaluation map is equal to that subset. The complexity of a language is determined by the minimal complexity class of machines able to solve the membership problem for that language. For example, languages recognized by finite state automata are called \emph{regular languages}. They are the simplest languages in a classification of formal languages called the Chomsky hierarchy. 

In this paper, we devise a criterion which we name \emph{language-convexity} for subgroups to inherit the regularity of a language representation of a group. Roughly speaking, a subgroup is language-convex with respect to a language $L$ if the prefixes of every word in $L$ represents an element in the subgroup, up to a bounded error. We apply this criterion to \emph{positive cones} of groups, which are the sets of elements greater than the identity under a left-invariant total order. 

\begin{theorem}\label{thm-fi} Let $G$ be a finitely generated group with a regular positive cone. If $H$ is a finite index subgroup, then $H$ also admits a regular positive cone.
\end{theorem}

A particularly simple class of regular languages is the class of finitely generated semigroups, which can be recognized by automata with only two states. The property of finite index subgroups inheriting a regular language representation of a positive cone is optimal in some sense, as refining this property to inheriting finite generation in the positive cone is impossible. Indeed, take the Klein bottle group given by presentation 
$$K_2 = \langle a, b : a\inv ba = b\inv \rangle.$$ It is easy to check that $K_2$ admits the positive cone $P = \langle a,b \rangle^+$. The subgroup given by $\langle a^2, b \rangle$, of index $2$, is isomorphic to $\ints^2$ since $$a^{-2}ba^2 = a^{-1}b\inv a = b.$$
It is a basic fact of orderability (see for example \cite[Section 2.2]{GroupsOrdersDynamics}) that a finitely generated positive cone of a group corresponds to an isolated point in the space of left orders of that group. Also well-known (see \cite[Section 1.2.1]{GroupsOrdersDynamics}) is the fact that the space of left orders on $\ints^2$ is isomorphic to the Cantor set, and thus cannot have any isolated points. Thus, $\ints^2$ does not admit a finitely generated positive cone, despite being a finite index subgroup of $K_2$. 

While finitely generated positive cones are easy to describe, not many examples of them are known. In his 2011 paper, Navas \cite{Navas2011} constructs an infinite family of groups given by presentation $\Gamma_n = \langle a, b : ba^nb = a \rangle$ which have positive cones of rank $2$. The author then poses the following problem: for every $k \geq 3$, find an infinite family of groups which admit a positive cone of rank $k$. We solve this problem completely by looking into finite-index subgroups of $\Gamma_n$.  
\begin{theorem}\label{thm-inf-fam} For every integer $m \geq 2$, and integer $n \geq 2 $ of the form $n = m-1 + mt$ for some odd integer $t$, there is a subgroup of index $m$ in $\Gamma_n = \langle a, b : ba^nb = a \rangle$ which admits a positive cone of rank $m+1$. 
\end{theorem}

In 2016, Hermiller and \v{S}uni\'{c} \cite{HermillerSunic2017NoPC} showed that no finitely generated free product has a regular positive cone. However, this property is not stable under taking a Cartesian product with the integers. For example, Rivas \cite{Rivas2018} had constructed an example of a regular positive cone for $F_2 \times \mathbb{Z}$. Furthermore, a 2018 result of Mann and Rivas \cite{Mann2018} states that this group has isolated points in its space of left orders. Since each finitely generated positive cone implies an isolated point in the space of left orders of a group (see for example \cite[Section 2.2]{GroupsOrdersDynamics}), this naturally leads to the question of whether $F_2 \times \ints$ admits a finitely generated positive cone. A special case of our Theorem \ref{thm-inf-fam} shows the following. 

\begin{corollary}\label{main-thm2} There exists a positive cone for $F_2 \times \ints$ which is finitely generated as a semigroup.
\end{corollary}

Moreover, we generalize the result of Hermiller and \v{S}uni\'{c} \cite{HermillerSunic2017NoPC} to acylindrically hyperbolic groups, which are a generalization of the class of non-elementary hyperbolic groups. Some examples of acylindrically hyperbolic groups are mapping class groups of closed, oriented surfaces, groups of outer automorphism of free groups, and free products. Our result postdates \cite{Calegari2002} that of Calegari, who showed in 2003 that no fundamental group of a hyperbolic manifold has a regular geodesic positive cone. 

\begin{theorem}\label{main-thm} A quasi-geodesic positive cone language of a finitely generated acylindrically hyperbolic group cannot be regular.
\end{theorem}

Since it is known that the lower bound of being 1-counter (the lowest complexity for a context-free language) is attained for some orders on free groups by a 2013 result of \v{S}uni\'{c} \cite{Sunic2013}, our bound is the best possible within the Chomsky hierarchy. In 2006, Farb posed the question of whether mapping class groups of closed, oriented surfaces of genus greater or equal to two, with either zero or one puncture, have a finite index subgroup which is left-orderable \cite[Problem 6.3]{Farb2006}. Our theorem makes partial progress on Farb's question by finding a lower bound on the formal language complexity for positive cones of acylindrically hyperbolic groups which are represented by languages containing only quasi-geodesic words. (Note that finite index subgroups of acylindrically hyperbolic groups are acylindrically hyperbolic \cite{Osin2013}.)

Our paper is structured as follows. We review some background in Section \ref{lang-bg}, which should be sufficient to understand the starting point of our paper: Hermiller and \v{S}uni\'{c}'s result \cite{HermillerSunic2017NoPC}. In Section \ref{fellow-travel}, we show how pairs of fellow-travelling words form a regular language. In Section \ref{convex}, we present the language-convexity criterion for subsets to inherit the property of having a regular language representation, then prove Theorem \ref{thm-fi}. The last two sections can be read independently of one another. In Section \ref{inf-fam}, we first provide an $(m+1)$-generated positive cone for certain subgroups of index $m$ of $\Gamma_n = \langle a,b : ba^nb = a \rangle$ for integers $n \geq 2, m \geq 2$. We then use this result to prove Corollary \ref{main-thm2}. Next, we show that $m+1$ is the minimal number of generators for the provided positive cones of these subgroup of index $m$ of $\Gamma_n$, for an infinite number of values of $n$, proving Theorem \ref{thm-inf-fam}. In Section \ref{acyl}, we apply our language-convexity criterion to known results in acylindrically hyperbolic groups to show that if there were  a regular quasi-geodesic positive cone language  for an acylindrically hyp{}erbolic group, it would allow us to construct a regular positive cone for the free group on two elements. This is a contradiction by the result of Hermiller and \v{S}uni\'{c} \cite{HermillerSunic2017NoPC} (Theorem \ref{herm-sunic}). 

\section{Background}\label{lang-bg}
The goal of this section is to present sufficient background to understand the starting point of this paper which is the statement of Theorem \ref{herm-sunic} due to Hermiller and \v{S}uni\'{c}. We end the section by briefly discussing Corollary \ref{main-thm2} and Theorem \ref{main-thm}. Should it be needed, we suggest additional reference \cite{GroupsOrdersDynamics} for orderability, and reference \cite{Epstein1992} for finite state automata.

\subsection{Left-orders and positive cones}\label{ord}
A group G is \emph{left-orderable} if it admits a strict total order $\prec$ of its elements such that the relation 
$g \prec h$ holds if and only if the relation $fg \prec fh$ holds for all $g,h$ and $f$ in $G$. Given an order $\prec$, the associated positive cone $P$ is the set of elements which are greater than the identity, $$P = \{g \in G : g \succ 1\}.$$$
P$ has two defining properties:
\begin{enumerate}
	\item $P$ is a semigroup in the sense that $PP \subseteq P$.
	\item $P$ defines a partition for $G$, $G = P \sqcup P\inv \sqcup \{1\}$ where the union is disjoint.
\end{enumerate} 

Equivalently, given a semigroup $P$ which partitions $G$ as above, we may define a left-invariant order $\prec$ by $$g \prec h \iff g\inv h \in P.$$ The induced positive cone by this left order is exactly $P$. Thus, the notions of left-order and positive cone are equivalent. 

It is straightforward to show that positive cones are closed under taking subgroups: if $H$ is a subgroup of $G$, then $H$ is also left-orderable with positive cone $P \intersect H$. 

\subsection{Regular languages}\label{reg-lang} A language over a finite set $X$ is a subset of the free monoid $X^* = \union_{n=0}^\infty X^n$, the set of arbitrarily long words with symbols in $X$. 

\begin{definition} A language is \emph{regular} if it is accepted by a finite state automaton.
\end{definition}

A \emph{finite state automaton} is a quintuple $\mathcal{A} = (S, X, \tau, A, s_0)$, where $S$ is a finite set called the \emph{state set}, $X$ is a finite alphabet for the input words, $\tau : S \times X \to S$ is a \emph{transition function} taking one state to another, $A \subseteq S$ is a set of states called the \emph{accept states} (or \emph{final states}), and $s_0 \in S$ is the \emph{initial state}. The function $\tau$ extends recursively to $\tau : S \times X^* \to S$ by setting $\tau(s, wx) = \tau(\tau(s,w),x)$ where $w \in X^*, x \in X$, and $s \in S$. The \emph{accepted language} by the automaton is the set of words $$\{w \in X^* : \tau(s_0, w) \in A\}.$$

\begin{remark}\label{fg-rmk}
Let $X$ be a finite alphabet, and set $P = \langle X \rangle^+$ to be a semigroup generated by $X$. We remark that $P$ is always accepted by an automaton $\mathcal{A}$ with two states. Indeed, set the states to be $S = \{s_0, s_1\}$ and the accept state to $A = \{s_1\}$, and define $\tau$  such that $$\tau(s_0, x) = s_1, \quad \tau(s_1, x) = s_1,\qquad \forall x \in X.$$ This automaton does not accept the empty word, since the inital state $s_0$ is not an accept state. However, it accepts any non-empty string in the semigroup, which is represented by always being in the accept state $s_1$ after one input letter. This shows that the property of being finitely generated as a semigroup is stronger than being a regular language.\end{remark}

Given two alphabets $X$ and $Y$ and a function $f : X \to Y^*$ there is a unique \emph{monoid homomorphism} $h: X^* \to Y^*$ extending $f$, namely, the map sending $w \in X^*$, $w = x_1 \dots x_n$ to $h(w) = f(x_1)\dots f(x_n)$, where each $f(x)$ for $x \in X$ is a word in $Y^*$. The image of a regular language $L \subseteq X^*$ under a monoid homomorphism is also regular. For a group $G = \langle X \rangle$, the \emph{evaluation map} $\pi: X^* \to G$ is the monoid homomorphism sending words in $X$ to the element they represent in $G$. Note that by abuse of notation, we do not distinguish between a generating set for $G$ and the set of symbols used to represent it. 

The following definition will be useful throughout our paper. 

\begin{definition}[Regular positive cone]
We say that a positive cone $P$ of a finitely generated group $G$ is a \emph{regular positive cone} if there exists a finite generating set $X$ and a regular language $L \subset X^*$ such that $\pi(L) = P$, where $\pi$ is the evaluation map.
\end{definition}

\begin{theorem}[Hermiller and \v{S}uni\'c \cite{HermillerSunic2017NoPC}]\label{herm-sunic}
Let $A, B$ be two non-trivial, finitely generated, left-orderable groups. Let $G = A * B$. Then $G$ does not admit a regular positive cone. 
\end{theorem}
In particular, this theorem states that $F_2$ cannot admit a regular positive cone. However, our Corollary \ref{main-thm2} states that there is a positive cone for $F_2 \times \ints$ which is finitely generated as a semigroup. In other words, taking the Cartesian product of $F_2$ with $\ints$ allows the resulting group to have a finitely generated positive cone, an even stronger property than having a regular positive cone.  

Theorem \ref{main-thm} generalizes Theorem \ref{herm-sunic} up to quasi-geodesic positive cones (see Definition \ref{def-qgpc}) in a class of groups containing free products, called acylindrically hyperbolic (see Section \ref{acyl-hyp-grp}). Both proofs will depend on the language-convexity criterion (see Proposition \ref{convex-crit}).

\section{Pairs of Fellow-Travelling Words Form A Regular Language} \label{fellow-travel}
We will prove that the set of pairs of words which represent the same element and asynchronously $M$-fellow-travel for some $M \geq 0$ is a regular language. This construction will play a key part in the proof of Proposition \ref{convex-crit}. In this section, fix $X$ to be a finite alphabet closed under formal inversion, and fix $\$$ to be a padding symbol. Define $X^\$ := X \union \{\$\}$. Let $G$ be a group generated by $X$, and let $\Gamma$ be the associated Cayley graph with vertex set $V(\Gamma)$ and word metric $d$. Let $\pi: (X^\$)^* \to G$ be the evaluation map which maps $x \in X$ to itself and $\$$ to the identity. As a shorthand, we will often denote $\pi(w)$ by $\bar w$.


\begin{definition}[Synchronous and asynchronous fellow-travel]
Let $(u,v) \in \Y$. We can identify $\Y$ with the subset of $(X^\$)^* \times (X^\$)^*$ consisting of pairs of words $(u,v)$ where $u \in \X^*, v \in \X^*$ and $u$ and $v$ have the same length. Let $u = x_1 \dots x_n$ and $v = y_1 \dots y_n$, and $u_i := x_1 \dots x_i$, $v_i := y_1 \dots y_i$ be the prefixes of $u$ and $v$, respectively. Let $M$ be a fixed non-negative constant. We say that $u$ and $v$ \emph{synchronously} $M$-fellow-travel if $d(\bar u_i, \bar v_i) \leq M$ for $i = 0,\dots,n$. 
We say that $u$ and $v$ \emph{asynchronously} $M$-fellow-travel if there exists a synchronously $M$-fellow-travelling pair $(u', v') \in \Y$ such that $u'$ and $v'$ are obtained by inserting padding symbols $\$$ between the symbols of $u$ and $v$. We will refer to $u'$ and $v'$ as the \emph{padded version} of $u$ and $v$, respectively.
\end{definition}

\begin{proposition}\label{L-aft}
Let $M \geq 0$. The language of pairs of words $(u,v) \in (\X \times \X)^*$ such that $u$ and $v$ synchronously $M$-fellow-travel and represent the same element in $G$, $$\L_M := \{(u,v) \in \Y : \bar u = \bar v \text{ and } d(\bar u_i, \bar v_i) \leq M, i = 1,\dots,n \}$$ is a regular language. 
\end{proposition} \label{asft-thm}

We remark that if $(u,v)$ is a pair of words over the unpadded alphabet $(X \times X)^*$, and $(u,v)$ asynchronously $M$-fellow-travel and represent the same element in $G$, then there are padded versions $u'$ and $v'$ of $u$ and $v$, respectively, such that $(u',v') \in \L_M$. Roughly speaking, $u'$ and $v'$ are the versions of $u$ and $v$ where one word ``waits'' for the other after each symbol by virtue of the placement of $\$$ between two symbols of $X$. In that sense, $\L_M$ captures the language of asynchronously $M$-fellow-travelling words in $(X \times X)^*$.
\begin{proof}[Sketch of the proof of Proposition \ref{L-aft}]
Define the finite-state automaton $\mathcal{A}$ as follows. The automaton $\mathcal{A}$ is the quintuple $(S, X^\$ \times X^\$, \tau, A, s_0)$, where $B_M \subseteq V(\Gamma)$ is the set of group elements contained in a ball of radius $M$ around the identity, and $S := B_M \union \{\rho\}$, where $\rho$ denotes a fail state. Let $g \in B_M$, and define the transition function $\tau: S \times (\X \times \X)^* \to S$, as

\begin{align*}
& \tau(g, (x,y)) = \begin{cases} \bar{x}\inv g \bar{y} & \bar{x}\inv g \bar{y} \in B_M\\ 
\rho &  \bar{x}\inv g \bar{y} \notin B_M \end{cases},  \qquad g \in B_M\\
& \tau(\rho, (x,y)) = \rho \qquad\qquad\qquad\qquad\qquad\qquad \forall (x,y) \in \Y.
\end{align*}

Let the accepting set of states be $A = \{1\}$, and the initial state be $s_0 = 1$. It is now straightforward to check by induction that this automaton accepts exactly the language $\L_M$.
\end{proof}

A proof of the following well-known theorem can be found in \cite[Section 1.4]{Epstein1992} and involves constructing the appropriate automata.
\begin{theorem}[Predicate calculus]\label{prop-cal}
Given regular languages $L_1$ and  $L_2$ over the same finite alphabet $Y$, the following languages are also regular.

\begin{itemize}
\item $L_1 \times L_2$ where  $L_1 \times L_2 = \{(u,v)\in (Y \times Y)^* : u\in L_1, v\in L_2\}$.
\item $L_1 \intersect L_2$. 
\end{itemize}
Moreover, if $L_3$ is a regular language over a product of finite alphabets $Y_1 \times \dots \times Y_n$ and $\text{Proj}_i : (Y_1 \times \dots \times Y_n)^* \to Y_i^*$ is the projection map on the $i$th coordinate,  then $\text{Proj}_i(L_3)$ is a regular language.

\end{theorem}
 We repeatedly use predicate calculus to prove the next lemma. 
  
 \begin{definition}[Padded language]\label{padded-lang}
Let $L$ be the regular language accepted by the finite state automaton $\mathcal{A} = (S, X, \tau, A, s_0)$. Fix $\$$ as a padding symbol. The \emph{padded language} $L^\$$ of $L$ is the language accepted by the automaton $\mathcal{A^\$} = (S, X, \tau^\$, A, s_0)$, where $\tau^\$$ is the function from $S \times (X \union \{\$\})$ to $S$ defined as
$$\tau^\$(s,x) := \begin{cases} \tau(s,x) & s \in S, x \in X \\ 
s & s \in S, x = \$. \end{cases}$$
The language $L^\$$ consists of all padded versions of the words in $L$, and is regular by construction.
\end{definition}

\begin{lemma}\label{L-tilde} Let $L \subseteq X^*$ be a regular language, let $M\geq 0$, and let $\mathcal{L}_M$ be the language of synchronously $M$-fellow-travelling pairs of words in $\Y$ such that in each pair $(u,v)$, the words $u$ and $v$ represent the same element in $G$, as defined in Proposition \ref{L-aft}. Then, $$\tilde L := \{v \in \Z 
: \exists u \in L^\$ \text{ such that } (u,v) \in \mathcal{L}_M\},$$ is a regular language, and $\pi(\tilde L) = \pi(L)$.
\end{lemma}

Note that $\tilde L$ is a language of words which synchronously $M$-fellow-travel with words in the padded language $L^\$$ and represent the same elements of $G$ as the words in $L$.

\begin{proof}[Proof of Lemma \ref{L-tilde}] We will be using Theorem \ref{prop-cal} several times. Let $$L' := \{(u,v) \in \Y : u \in \padL \text{ and } v \in \Z\}.$$ Observe that $L' = L^\$ \times (\X)^*$ so it is regular. Set $$L'' :=  \{(u,v) \in \Y : u \in \padL \text{ and } (u,v) 
\in \mathcal{L}_M\}.$$ Since $L'' = \L_M \intersect L'$, so $L''$ is also regular. Set $$
\tilde L := \{v \in \Z 
: \exists u \in \padL \text{ such that } (u,v) \in \mathcal{L}_M\} $$
and observe that $\tilde L = \text{Proj}_2(L'')$, so it is regular. Finally, we observe that 
\begin{align*}
\pi(L) 
&= \pi(\text{Proj}_1(L')), & \qquad \pi(L) = \pi(L^\$) \\
&= \pi(\text{Proj}_1(L'')), & \qquad (u,v) \in L' \imp (u,u) \in L'' \\
&= \pi(\text{Proj}_2(L'')), & \qquad (u,v) \in L'' \imp \pi(u) = \pi(v)\\
&= \pi(\tilde L). \end{align*} \end{proof}

\section{Language Convex Subgroups}\label{convex}
This section is dedicated to the definition of language-convexity and the proof of Theorem \ref{thm-fi}. For this section, let $X$ be a finite alphabet which is closed under formal inversion and let $G = \langle X \rangle$ be a group generated by $X$. Let $\Gamma$ be the associated Cayley graph with the graph metric. For $w$ in $X^*$,  $w = x_1 \dots x_n$, and for $i = 1,\dots,n$, we will denote by $w_i$ the prefix of length $i$ of $w$, i.e. $w_i := x_1 \dots x_i$. Let $|w|$ denote the length of $w$.  Set $\pi : X^* \to G$ as the evaluation map. For convenience, we will denote $\pi(w)$ by $\bar w$. 

\subsection{Word-induced paths}\label{word-ind-path} Every word in $(\X)^*$ induces a path in the Cayley graph $\Gamma$. Indeed, assume that we have a word $w \in X^*$. Let $w = x_1 \dots x_n$ for $x_i \in X$ and $i = 1,\dots,n$.  The induced path $p_w = (1 = \bar w_0, \dots, \bar w_{n-1}, \bar w_n)$ is a sequence composed of vertices $\bar w_i$ which are the elements of $G$.  The path $p_w$ can be parametrized in the natural way as a continuous function $p_w: [0,n] \to \Gamma$ satisfying $p_w(i) = \bar u_i$ for $0 \leq i \leq n$. On the other hand, if $w' \in (\X)^*$, then take $w'$ as inducing the same sequence of vertices as its unpadded version. We view the path $p_{w'}$ as unparametrized.

\begin{definition}[language-convexity]\label{lconvex}
Let $L$ be a language over $X$. A subset $H \subseteq G$ is \emph{language-convex with respect to $L$} if there exists an $R \geq 0$ such that for each $w \in L$ with $\bar w \in H$, the induced path $p_w$ lies within distance $R$ of $H$ in $\Gamma$.
\end{definition}

\begin{proposition}[language-convexity criterion]\label{convex-crit} Let $X$ be a finite set which is closed under formal inversion, $X = X^{-1}$. Set $G = \langle X \rangle$. Let $L$ be a regular language, and let $P = \pi(L)$ where $\pi$ is the evaluation map onto $G$. Let $H$ be a subgroup of $G$. If $H$ is language-convex with respect to $L$, then there exists a regular language $L_H$ such that  $\pi(L_H) = H \intersect P$. 
\end{proposition}

This definition will be useful for the proof of the above proposition. 

\begin{definition}[Geodesic words]\label{geodesic-words}
A \emph{geodesic path} is a path $p$ connecting two vertices $v_1, v_2 \in V(\Gamma)$ such that $p$ has the shortest length amongst all paths from $v_1$ to $v_2$ in $\Gamma$. A word $w \in (\X)^*$ is \emph{geodesic} if $w \in X^*$ and the induced path $p_w$ is geodesic. If $g$ is a group element such that $\bar w = g$, then $w$ is a \emph{geodesic representative} of $g$. 
\end{definition}
{}
\begin{proof}[Proof of Proposition \ref{convex-crit}]

Let $L$ and $P$ be as in the statement of Proposition \ref{convex-crit}. Suppose that $H$ is language-convex with respect to $L$. We want to show that $H \intersect P$ can be represented by a regular language.

Let $R$ be, as in Definition \ref{lconvex}, the convexity parameter for $H$. Set $$Y = \{y \in H : |y| \leq 2R + 1\}.$$ Let $\vphi: Y^* \to X^*$ be the monoid homorphism sending each element $y \in Y$ to a geodesic representative in terms of the finite generating set $X$. Fix $M = 3R + 1$. Let $\tilde L$ be the regular language given by Lemma \ref{L-tilde}, which consists of the set of padded words in $X^{\$}$ which synchronously $M$-fellow-travel with words in the padded language $L^{\$}$. Recall that $\pi(\tilde L) = \pi(L)$. Set $$L_H = \vphi(Y^*) \intersect \tilde L.$$ The regularity of $L_H$ is given by Theorem \ref{prop-cal}. We will argue that $L_H$ is a language representing $H \intersect P$. 

We start by showing that $\pi(L_H) \supseteq H \intersect P$. Let $g \in H \intersect P$. Set $u = x_1 \dots x_n$ to be a representative of $g$ in $L$. Since $\bar u \in H$, we have by language-convexity that for each prefix $u_i = x_1 \dots x_i$, the evaluation $\bar u_i$ is at distance at most $R$ from $H$. Therefore, there exists $h_i \in H$ satisfying $d(\bar u_i, h_i) \leq R$ in $\Gamma$. Since $\bar u_0 = 1$ and $\bar u_n \in H$, we are allowed to set $h_0 = 1$ and $h_n = \bar u_n$. For each $i \in \{1,\dots,n\}$, let $y_i = h_{i-1}\inv h_i$. Observe that $y_i \in Y$. Indeed, $$|y_i| = d(h_{i-1}, h_i) \leq d(h_{i-1}, \bar u_{i-1}) + d(\bar u_{i-1},\bar u_i) + d(\bar u_i, h_i) \leq 2R + 1.$$

Let $w = y_1 \dots y_n$ and $w_i = y_1 \dots y_i$. Let $v = \vphi(w)$. It is clear that $\bar v = h_n = g$, and that $v$ belongs to the monoid $\vphi(Y^*)$. 

To show that $v \in \tilde{L}$, we will to show that it asynchronously $M$-fellow-travels with $u$. First observe that each $\bar u_i$ is at distance at most $R$ from each $\overline{\vphi(w_i)}$ by construction of $v$. Recall that the geodesic subpath connecting $\overline {\vphi (w_{i-1})}$ to $\overline {\vphi (w_{i})}$ is labelled by $\vphi(y_i)$ for $i = 1,\dots\,n$, where $w_0$ is the empty word. Therefore, any vertex in such a subpath is at distance at most $2R + 1$ from $\overline {\vphi (w_{i})}$ and hence at most $M = 3R + 1$ from $\bar u_i$. Let $u'$ be the padded word for $u$ which has $|y_i| - 1$  padding symbols $\$$ added after each $x_i$ for $i = 1,\dots,n$. Then $(u',v)$ synchronously $M$-fellow-travel. This shows that $v \in \tilde L$, and thus $\pi(L_H) \supseteq H \intersect P$. 

To conclude the proof, we show that $\pi(L_H) \subseteq H \intersect P$. If $w \in L_H$, then $\bar w \in \pi(\vphi(Y^*))$. Since $\pi(\vphi(Y)) \subseteq H$, we have that $\bar w \in H$. Moreover, $\bar w \in \pi(\tilde L) = \pi(L) = P$, so we obtain that $\pi(L_H) \subseteq H \intersect P$.
\end{proof}

For a language $L$ over $X^*$, denote by $\mathcal{A}(L)$ an automaton accepting the language $L$ such that it has a minimal number of states amongst all automata accepting $L$. Denote by $|\mathcal{A}(L)|$ the number of states of $\mathcal{A}(L)$. We conclude this section by getting an estimate for $|\mathcal{A}(L_H)|$, where $L_H$ is as in Proposition \ref{convex-crit}.

Let $\gamma_H$ denote the growth function of the subgroup $H \leq G$ with respect to the generating set $X$, that is $$\gamma_H(n) = | \{g \in H : |g| \leq n \} |.$$ 

\begin{corollary}
Let $G, X$ and $L$ be as in Proposition \ref{convex-crit}. Let $H$ be language-convex with respect to $L$ with convexity parameter $R$. There is a regular language $L_H$ representing $H \cap \pi(L)$ such that $$|\mathcal{A}(L_H)| \leq (2R + 1) \cdot |\mathcal{A}(L)| \cdot \gamma_H(2R + 1) \cdot (\gamma_G(3R + 1) + 1) .$$

\end{corollary}
\begin{proof}
Recall from the proof of Proposition \ref{convex-crit} that we can take $L_H = \vphi(Y^*) \cap \tilde L$, where $\tilde L$ is given by Lemma \ref{L-tilde}, and $Y = \{h \in H : |h| \leq 2R + 1\}$. Recall that for $y \in Y$, $\vphi(y)$ is a geodesic in $X$ representing $y$. 
It is easy to show that $|\mathcal{A}(\vphi(Y^*))| \leq (2R + 1)\cdot \gamma_H(2R + 1)$. The states of an automaton accepting an intersection of two regular languages are given by the product of the states of the two automaton accepting each of the languages. Therefore, we have $$|\mathcal{A}(L_H)| \leq (2R + 1) \cdot \gamma_H(2R + 1) \cdot |\mathcal{A}(\tilde L)|.$$

It remains to bound $|\mathcal{A}(\tilde L)|$. Set $M = 3R + 1$. Set $\mathcal{L}_M$ to be the language in Proposition \ref{L-aft}. It follows from the proof of Proposition \ref{L-aft} that $\mathcal{A} (\mathcal{L}_M)$ has at most $\gamma_G(3R + 1) + 1$ states. Finally, set $\tilde L, L', L''$ and $L^\$$ to be as in Lemma \ref{L-tilde}. Recall from the proof of Lemma \ref{L-tilde} that $L' = L^\$ \times (\X)^*, L'' = \L_M \intersect L'$ and $\tilde L = \text{Proj}_2(L'')$. Since taking a projection does not increase the number of states, $|\mathcal{A}(\tilde L)| \leq |\mathcal{A}(L'')|$. By the previous remark about intersection, $|\mathcal{A}(L'')| \leq |\mathcal{A}(L')| \cdot |\mathcal{A}(\mathcal{L}_M)|$. One can construct an automaton for $L' = L^\$ \times (\X)^*$ by taking the product of automata for $L^\$$ and $(\X)^*$. By Remark \ref{fg-rmk} and Definition \ref{padded-lang}, $|\mathcal{A}((\X)^*)| = 1$ and $|\mathcal{A}(L^\$)| = |\mathcal{A}(L)|$. We conclude that $|\mathcal{A}(L')| \leq |\mathcal{A}(L)|$. Putting it all together, the corollary follows. 
\end{proof}

We will now use our language-convexity criterion to show that finite index subgroups are language-convex with respect to any language. We will then apply our result in the proof of Corollary \ref{main-thm2} to show that there is a regular positive cone for $F_2 \times \ints$. We will then show a stronger version of that result by independently constructing a positive cone for $F_2 \times \mathbb{Z}$ which is finitely generated as a semigroup. 

\begin{lemma}\label{finite-index}
Let $H$ be a language-convex subgroup of a finitely generated group $G$ with respect to a language $L$. If $K$ is a finite index subgroup of $H$, then $K$ is also language-convex subgroup of $G$ with respect to $L$.
\end{lemma}

\begin{proof}
Let $R$ be the language-convexity constant of $H$ with respect to $L$. Assume that  $K$ is a finite index subgroup of $H$. Let $C = \{h_1, \dots, h_n\}$ be a list of coset representatives of $K$ in $H$. Let $R' = \max_{h_i \in C} |h_i|$.  If $h \in H$, then there exists an $i \in \{1, \dots, n\}$ such that $h = kh_i$. Then $d(h,k) = |h_i\inv k\inv k| = |h_i| \leq R'$. This shows that $H \subseteq N_{R'}(K)$. 

Then for all $w \in L$ with $\bar w \in K$, we have that $p_w \subseteq N_R(H)$ by language-convexity of $H$. Moreover, $N_R(H) \subseteq N_R(N_{R'}(K)) = N_{R + R'}(K)$. Therefore $K$ is language-convex with convexity parameter $(R + R')$.
\end{proof}
Theorem \ref{thm-fi} is a corollary of this lemma. 

A particularly nice application of this lemma is on the braid group on three strands, which admits a regular positive cone. This group has a finite index subgroup isomorphic to $F_2 \times \ints$, which inherits a regular positive cone by Lemma \ref{finite-index}.  We will review the this material in the proof of Corollary \ref{main-thm2}.

\section{Constructing an infinite family of groups with $k$-generated positive cones}\label{inf-fam}

Consider the group $\Gamma_n = \langle a,b : ba^nb = a \rangle$. Let $\Delta = a^{n+1}$. Note that $\Delta$ is central in $\Gamma_n$. Indeed, 
\begin{align*}
b\Delta &= ba^{n+1} 
= (ba^n)a 
= (ab\inv)a 
= a(b\inv a) 
= a(a^n b) 
= \Delta b. 
\end{align*}
By \cite{Navas2011}, $\Gamma_n$ admits a positive cone $P_n = \langle a,b \rangle^+$ for all integer $n \geq 1$. 
The following lemma will be useful in the study of positive cones of subgroups of $\Gamma_n$. 

\begin{lemma}\label{in-P}
For all integer $n \geq 1$, the element $b^{-s}a$ where $s \geq 0$ belongs to $P_n$.
\end{lemma}
\begin{proof}
We will show this by induction on $s$ that $b^{-s}a = a(a^{n-1}b)^s$. If $s=0$, then $$b^0a = a = a(a^{n-1}b)^0.$$

As for the $s \imp s+1$ case, 
\begin{align*}
b^{-(s+1)}a &= b^{-s}b\inv a \\
&= a(a^{n-1}b)^s a\inv (b\inv a) & (a^n b = b\inv a)\\
&= a(a^{n-1}b)^s a\inv a^nb\\ 
&= a(a^{n-1}b)^s a^{n-1}b \\
&= a(a^{n-1}b)^{s+1}.
\end{align*}
Since $P_n = \langle a, b \rangle^+$, the element $b^{-s}a = a(a^{n-1}b)^{s}$ clearly belongs to $P_n$. 
\end{proof}

We are now going to look at a particular class of finite-index subgroups of $\Gamma_n$. Let $n,m$, and $\mu$ be such that 
$$(n-1)\mu \equiv -2 \mod m.$$
Note that given a fixed pair $n$ and $m$, there may not necessarily be a solution for $\mu$. For each triple $n,m$, and $\mu$ satisfying the equation above, define a homomorphism $\vphi: \Gamma_n \to \mathbb{Z}/m\mathbb{Z}$ by setting $\vphi(a) = \mu$ and $\vphi(b) = 1$. We check that the map is a homomorphism by verifying that the relation is satisfied in the image $$\vphi(ba^nba\inv) \equiv 2 + (n-1)\vphi(a) \equiv 0 \mod m.$$

\begin{proposition}\label{pos-cone-gen}
Let $n,m$ and $\mu$ be such that $(n-1)\mu \equiv -2 \mod m$. Let $\vphi: \Gamma_n \to \mathbb{Z}/m\mathbb{Z}$ be a homomorphism such that $\vphi(a) = \mu$ and $\vphi(b) = 1$. Let $H := \ker \vphi$ and let $P = \langle a , b \rangle^+$ be a positive cone for $\Gamma_n$ (which is proven in \cite{Navas2011}). Then $H \cap P$ admits the finite generating set $Y$, where $$Y := \{b^{-s}ab^{s+(m-\mu)}\}_{s=0}^{\mu-1} \cup \{b^{-s}ab^{s-\mu}\}_{s=\mu}^{m-1} \union \{b^m\}.$$
\end{proposition} 

The proof of this proposition will rely on the Reidemeister-Schreier method, which we recall below. 

\begin{definition}[Schreier transversal] Let $F$ be a free group, and $\tilde H$ be a subgroup of $F$. A \emph{Schreier transversal} $T$ of $\tilde H$ is a subset of $F$ such that for distinct $t \in T$, the cosets $\tilde Ht$ are distinct, $\bigcup_{t \in T} \tilde Ht = F$, and such that each initial segment (prefix) of an element of $T$ belongs to $T$. 
\end{definition}

\begin{proposition}[Reidemeister-Schreier method \cite{LyndonSchupp2001}]\label{RS}
Let $G = F/N$, where $F$ is free with basis $X$ and $N$ is the normal closure of the relator set $R$. Let $\phi: F \to F/N$ be the natural map of $F$ onto $G$. Let $H$ be a subgroup of $G$ with $\tilde H$ as the inverse image under $\phi$, and let $T$ be a Schreier transversal for $\tilde H$ in $F$. For $w$ in $F$, we define $\bar w$ by the condition that 
$$\tilde Hw = \tilde H\bar w, \quad \bar w \in T.$$
For $t \in T$, $x \in X$, we define 
$$\gamma(t,x)^* = tx(\ovl{tx})\inv.$$

Define a one-to-one correspondence between $\gamma(t,x)^*$ and $\gamma(t,x)$. Then $H$ has presentation $\langle Y \mid S \rangle$ where $Y = \{\gamma(t,x) : t \in T, x \in X, \gamma(t,x)^*\not= 1\}$. 

Let $F'$ be the free group generated with basis $Y$. Define $\tau: F \to F'$ as follows. If $w = y_1 \dots y_\ell$, 
$$\tau(w) = \gamma(1,y_1) \dots \gamma(\ovl{y_1 \dots y_{i-1}}, y_i) \dots \gamma(\ovl{y_1 \dots y_{\ell-1}}, y_\ell).$$
Then $S = \{\tau(trt\inv) : t \in T, r \in R\}$.
\end{proposition}

\begin{lemma}\label{Y-as-gen} 
Let $\gamma$ and $\gamma^*$ be as in Proposition \ref{RS}, and $\vphi$, $\mu$, and $H$ be as in Proposition \ref{pos-cone-gen}. Then $$Y = \{b^{-s}ab^{s+(m-\mu)}\}_{s=0}^{\mu-1} \cup \{b^{-s}ab^{s-\mu}\}_{s=\mu}^{m-1} \union \{b^m\}$$
generates $H$. 
\end{lemma} 

\begin{proof}
Let $\phi: F_2 \to \Gamma_n$ be the canonical map from a free group on two elements onto $\Gamma_n$. The set $T = \{b^0, b^{-1}, \dots, b^{-(m-1)}\}$ is a Schreier transversal for $\tilde H := \phi \inv (H)$, since the restriction of $\vphi \phi: F_2 \to \mathbb{Z}/m\mathbb{Z}$ is bijective. By the Reidemeister-Schreier method, a generating set for $H$ is given by $\{\gamma(t,x): t \in T$, $x \in \{a,b\}, \gamma(t,x)^*\not=1\}$. Now,  

\begin{align*}
\gamma(b^{-s},a)^* 
&= \begin{cases} 
b^{-s}ab^{-s + \mu} & \text{ if } 0 \leq \mu \leq s  \\
b^{-s}ab^{m-(-s + \mu)} & \text{ if } s + 1 \leq \mu \leq m-1
\end{cases} \\
\gamma(b^{-s},b)^* 
&= \begin{cases} 
1 & \text{ if } 1 \leq s \leq m-1 \\
b^m & \text{ if } s = 0. \end{cases}
\end{align*}
By identifying $\gamma(t,x)$ with $\gamma(t,x)^*$ when it is not the identity, we obtain $Y$ as a generating set for $H$.
\end{proof}

\begin{proof}[Proof of Proposition \ref{pos-cone-gen}]
We will first show that $\langle Y \rangle^+ \subseteq H \cap P$.
We have shown in Lemma \ref{Y-as-gen} that $Y$ generates $H$, and by Lemma \ref{in-P} that $Y \subseteq P$. Thus $Y \subseteq H \cap P$. Since $H \cap P$ is a semigroup, $\langle Y \rangle^+ \subseteq H \cap P$. 

To show that $H \cap P \subseteq \langle Y \rangle^+$, we will show that for every word $w \in \langle a,b \rangle^+$ whose image $\pi(w)$ is in $H$, there is a corresponding word $v \in \langle Y \rangle^+$ such that $\pi( v) = \pi(w)$. 

Write $w = x_1 \dots x_{\ell}$ and let $w_i = x_1 \dots x_i$. Recall the map $\tau: F_2 \to F(Y)$ from the Reidemeister-Schreier method (Proposition \ref{RS}), where $F(Y)$ stands for the free group with basis $Y$. Since $\pi( w) \in H$, $\tau(w)$ is well-defined, and $\pi(\tau(w)) = \pi( w)$ by construction of $\tau$. Furthermore,
$$\tau(w) = \prod_{i=1}^\ell \gamma(\ovl{w_{i-1}},x_i).$$
Since $w \in \langle a,b \rangle^+$, $x_i \in \{a,b\}$ for $1 \leq i \leq \ell$. Thus $\gamma(\ovl{w_{i-1}}, x_i) \in Y \cup \{1\}$ for $1 \leq i \leq \ell$.

Define $v := y_1 \dots y_\ell$ where
$$y_i = \begin{cases} 
\gamma(\ovl{w_{i-1}},x_i) & \text{ if } \gamma(\ovl{w_{i-1}},x_i) \not=1 \\ 
\varepsilon & \text{ if } \gamma(\ovl{w_{i-1},x_i}) = 1 
\end{cases}$$

Then $v \in \langle Y \rangle^+$ and $\pi(v) = \pi(\tau(w)) = \pi( w)$. This shows that $H \cap P = \langle Y \rangle^+$.

\end{proof}

\begin{definition}[Rank] The rank of a finitely generated group (resp. semigroup) is the smallest size of a generating set needed to generate the group (resp. semigroup). 
\end{definition}
This corollary follows from Proposition \ref{pos-cone-gen}.
\begin{corollary}
Let $n,m$ and $\vphi$ be such that $(n-1)\vphi(a) \equiv -2 \mod m$. Let $H := \ker \vphi$ as in Proposition \ref{pos-cone-gen}, and let $P := P_n = \langle a , b \rangle^+$. Then the rank of $H \cap P$ is at most $m+1$. 
\end{corollary} 

Recall that Corollary \ref{main-thm2} states that $F_2 \times \ints$ admits a positive cone which is finitely generated as a semigroup. We will show that this fact is a corollary of Proposition \ref{pos-cone-gen}. 
\begin{proof}[Proof of Proposition \ref{main-thm2}]
Let $n=2, m = 6$. A possible solution for the equation 
$$2 + (n-1)\vphi(a) \equiv 0 \mod 6$$ 
is $\vphi(a) = 4$. Let $H = \ker \vphi$. Then by Proposition \ref{pos-cone-gen}, $H$ admits a positive cone generated by $$Y = \{ab^2, b^{-1}ab^3, b^{-2}ab^4, b^{-3}ab^5, b^{-4}a, b^{-5}ab, b^6\}.$$ A computation in GAP based on the Reduced Reidemeister-Schreier method \cite[Chapter 47]{GAP4} and Tietze transformations \cite[Chapter 48]{GAP4} reveals that $$H \cong F_2\times \ints.$$ Indeed, under the presentation given by $\Gamma_2 = \langle a,b : ba^2b = a \rangle$, we have that $H = \langle ab^2, a^2b^2a^2, a^3 \rangle$. Using the following presentation for $F_2 \times \mathbb{Z}$, $F_2 \times \ints = \langle x,y,z : [x,z] = [y,z] = 1\rangle$, the isomorphism map $\psi: H \to F_2 \times \mathbb{Z}$ is given by
$$\psi(ab^2) = x, \qquad \psi(a^2b^2a^2) = y, \qquad \psi(a^3) = z.$$
Making use of this map and the normal form provided in \cite[Section 1]{Navas2011}, we find that $$\psi(Y) = \{ x, yxz^{-1}, x^{-1}yxz^{-1}, x^{-1}y^{-1}x^{-1}yx, x^{-1}y\inv z^2, x^{-1}y^{-1}xz^2, y^{-1}x^{-1}yxz^{-1} \}.$$
\end{proof}
\begin{remark}
The group $\Gamma_2 = \langle a,b : ba^2b = a \rangle$ is isomorphic to the braid group given by $B_3 = \langle \sigma_1, \sigma_2 : \sigma_1 \sigma_2 \sigma_1 = \sigma_2 \sigma_1 \sigma_2 \rangle$ by identifying $a \mapsto \sigma_1 \sigma_2, b \mapsto \sigma_{2}{\inv}$. 
\end{remark}

So far, we have shown that for certain integers pairs $n \geq 2$ and $m \geq 2$, there exists a homomorphism $\vphi: \Gamma_n \to \mathbb{Z}/m\mathbb{Z}$ which creates a subgroup $H := \ker \vphi$ of index $m$ which admits a positive cone with at most $m+1$ generators. In the sequel, we will show that for every fixed $m$, it is possible to pick an infinite family of $\Gamma_n$'s satisfying a certain criterion on $n$ such that the positive cone of the subgroup $H$ in question has a minimal number of generators that is exactly $m+1$. To aid our proof, we will use the following lemma.

\begin{lemma}\label{rk} Let $G$ be an left-orderable group, and let $P$ be a positive cone of $G$ generated by $r$ elements. If $G$ has rank at least $r$, then $r$ is the rank of $P$. 
\end{lemma}
\begin{proof}
Clearly, the rank of $P$ is at most $r$. Suppose $P$ can alternatively be generated by a finite set of of cardinality $k$, $P = \langle x_1, \dots, x_k \rangle^+$, then since $G = P \sqcup P\inv \sqcup \{1\},$ we have that $G = \langle x_1, \dots, x_k \rangle$. Therefore, $k \geq r$. \end{proof}

\begin{proposition}\label{H-pres}
Let $n = m-1 + mt$, where $t$ is a non-negative integer. Let $\Gamma_n = \langle a,b \mid ba^nb = a \rangle$. Then $a \mapsto 1, b \mapsto 1$ extends to a surjective homomorphism $\vphi: \Gamma_n \to \mathbb{Z}/m\mathbb{Z}$. The group $H := \ker \vphi$ is a subgroup of $\Gamma_n$ of index $m$ admitting a presentation on $m+1$ generators and $m$ relators, $H = \langle x_0, \dots, x_m \mid S \rangle$ where 
$$S = \{ x_i x_m^{t+1} x_i \mid i = 0, \dots, m-2 \} \cup \{x_{m-1}x_m^t x_{m-1}x_m\inv\}.$$ 
We may embed the generators of $H$ into $\Gamma_n$ by sending $x_i \mapsto a^iba^{-i-1}$ for $i = 0, \dots, m-2$, $x_{m-1} \mapsto a^{m-1}b$, and $ x_m \mapsto a^m$. 
\end{proposition}

\begin{corollary}\label{Ab-H}
Let $m,n,t$ and $H$ be as defined in Proposition \ref{H-pres}. If $t$ is an odd integer, then the abelianization of $H$, $H_{\text{ab}}$, is isomorphic to $(\mathbb{Z}/2\mathbb{Z})^m \times \mathbb{Z}$.
\end{corollary}
\begin{proof}[Proof of Corollary \ref{Ab-H}]
Take the presentation of $H$ as given in Proposition \ref{H-pres}, and make a natural identification $x_i \mapsto y_i$, for $i=0, \dots, m$ from the generators of $H$ to the generators of the abelianization of $H$. Then, $H_{\text{ab}}$ has a presentation with generators $\{y_0, \dots, y_m\}$ and relators $\{2y_i + (t+1)y_m : 0 \leq i \leq m-2\} \cup \{2y_{m-1} + (t-1)y_m \}$.
Assume $t$ is odd, and define $z_i := y_i + \frac{t+1}{2}y_m$ for $i = 0, \dots, m-2$, and $z_{m-1} := y_{m-1} + \frac{t-1}{2}y_m$. By Tietze transformations, we may rewrite the presentation of $H_{{\text{ab}}}$ as 
$$H_{{\text{ab}}} = \langle z_0, \dots, z_{m-1}, y_m \mid z_0^2, \dots, z_{m-1}^2 \rangle,$$
from which we can clearly see that $H_{\text{ab}}$ isomorphic to $(\mathbb{Z}/2\mathbb{Z})^m \times \mathbb{Z}$. 
\end{proof}

\begin{corollary}\label{ab-rank}
Let $m,n,t$ and $H$ be defined as in Proposition \ref{H-pres}. If $t$ is an odd integer, then $H$ has a positive cone of rank $m+1$
\end{corollary}
\begin{proof}
By Corollary \ref{Ab-H}, the subgroup $H$ has an abelianization of rank $m+1$. By Proposition \ref{pos-cone-gen}, $H$ admits a positive cone $P$ with $m+1$ generators. By Lemma \ref{rk}, $m+1$ is the rank of $P$. 
\end{proof}
The previous corollary shows that for any $m+1 \geq 3$, the subgroup $H \leq \Gamma_n$ as defined in Proposition \ref{H-pres} has a positive cone of rank $m+1$ as long as $n$ is of the form $n = m-1 + mt$ with odd integer $t \geq 1$. Therefore, the family $$\{H \leq \Gamma_n \mid n = m-1 + mt, \quad t \geq 1 \text{ and is odd}\}$$ satisfies the statement of Theorem \ref{thm-inf-fam}. 
We will now prove Proposition \ref{H-pres}. 

\begin{proof}[Proof of Proposition \ref{H-pres}]
Let $\phi: F_2 \to \Gamma_n$ be the {}canonical map. We will be using the Reidemeister-Schreier method again, this time with choice of transversal $T = \{1, a, a^2, \dots, a^{m-1}\}$. Our transversal $T$ is a Schreier transversal since the restriction $\vphi \phi: F_2 \to \mathbb{Z}/m\mathbb{Z}$ is also bijective as now $a \mapsto 1$. Recall the functions $\gamma^*, \gamma$ and $\tau$ from Proposition \ref{RS}. We know that $H$ is generated by $\{\gamma(t,x) :  t \in T, x \in \{a,b\}, \gamma(t,x)^*  \not= 1\}$.

Now,
\begin{align*}
\gamma(a^s, a)^* &= \begin{cases} 
1 & 0 \leq s \leq m-2 \\
a^m & s = m-1
\end{cases} 
\\
\gamma(a^s, b)^* &= \begin{cases}
a^s b a^{-(s+1)} & 0 \leq s \leq m-2 \\
a^{m-1} b & s = m-1.
\end{cases}
\end{align*}

Therefore, by identifying $x_s := \gamma(a^s,b)$ with $x_s$ for $s = 0, \dots, m-2$, $\gamma(a^{m-1}, b)$ with $x_{m-1}$ and $\gamma(a^{m-1},a)$ with $x_m$, the set $\{x_0, \dots, x_m\}$ generates $H$.

To compute the relators of $H$, recall that for a word $w = y_1\dots y_\ell$ with prefixes $w_i = y_1\dots y_i$, the function $\tau$ send $w$ to $\tau(w) = \prod_{i=1}^{\ell} \gamma(\overline{w_{i-1}}, y_{i})$. The first case of relators to compute are relators of the form $\tau(a^sba^{n}ba^{-(s+1)})$ where $s = 0, \dots, m-2$. Then,

$$\tau(a^sba^{n}ba^{-(s+1)}) = \left(\prod_{i=0}^{s-1} \gamma(a^i, a) \right) \gamma(a^s, b) \left( \prod_{i=s+1}^{s+n} \gamma(\overline{a^i}, a) \right)  \gamma(\overline{a^{s + n + 1}},b) \left( \prod_{i=0}^{s}\gamma(\overline{a^{s+n+2-i}}, a\inv) \right).$$


Observe that the first factor $\prod_{i=0}^{s-1} \gamma(a^i, a) = 1$. By replacing $n$ with $m-1+mt$, in the last factor, we obtain $\prod_{i=0}^{s}\gamma(\overline{a^{s+n+2-i}}, a\inv) = \prod_{i=0}^{s}\gamma({a^{s+1-i}}, a\inv) = 1$. We are left with 
$$\tau(a^sba^{n}ba^{-(s+1)}) = \gamma(a^s, b) \left( \prod_{i=s+1}^{n+s} \gamma(\overline{a^i}, a) \right)  \gamma(\overline{a^{s +m(t+1)}},b) = x_s \left( \prod_{i=s+1}^{m-1+mt+s} \gamma(\overline{a^i}, a) \right) x_s .$$


We claim by induction on $t$ that 
$\prod_{i=s+1}^{m-1+mt+s} \gamma(\overline{a^i}, a) = x_m^{t+1}.$
The base case $t=0$ gives us
\begin{align*}
\prod_{i=s+1}^{m-1+s} \gamma(\overline{a^i}, a) &=\underbrace{\gamma(a^{s+1},a) \gamma(a^{s+2}, a) \dots \gamma(a^{m-2},a)}_{=1} \underbrace{\gamma(a^{m-1},a)}_{=x_m} \underbrace{\gamma(\overline{a^m}, a) \dots \gamma(\overline{a^{m-1+s}},a)}_{=1} \\
&= x_m.
\end{align*} 
Assuming the hypothesis, 
\begin{align*}
\prod_{i=s+1}^{m-1+mt+s} \gamma(\overline{a^i}, a) &= \left(\prod_{i=s+1}^{m-1+m(t-1)+s} \gamma(\overline{a^i}, a)\right)\left(\prod_{i=mt+s}^{m-1+mt+s} \gamma(\overline{a^i}, a)\right) \\
&= x_m^{t} \left(\prod_{i=mt+s}^{m-1+mt+s}\gamma(\overline{a^i}, a)\right) \\
&= x_m^{t} \left(\prod_{i=s}^{m-1+s} \gamma(\overline{a^i}, a)\right) \\
&= x_m^{t} \underbrace{\gamma(a^s,a)}_{=1} \left(\prod_{s+1}^{m-1+s}\gamma(\overline{a^i},a)\right) \\
&= x_m^{t+1}.
\end{align*} 

Therefore, $\tau(a^sba^nba^{-(s+1)}) = x_s x_m^{t+1} x_s$ for $0 \leq s \leq m-2$. As for the $s=m-1$ case, 
$$\tau(a^{m-1}ba^nba^{-m})=\underbrace{\left(\prod_{i=0}^{m-2}\gamma(a^i, a)\right)}_{=1}\gamma(a^{m-1},b)\left(\prod_{i=m}^{m+n-1}(\gamma(\overline{a^i},a)\right)\gamma(\ovl{a^{m+n}},b)\left(\prod_{i=0}^{m-1}\gamma(\ovl{a^{m+n+1-i}},a\inv)\right)$$ 
We can simplify the last factor as follows, again replacing $n$ by $m-1+mt$.
$$\prod_{i=0}^{m-1}\gamma(\ovl{a^{m+n+1-i}},a\inv) = \prod_{i=0}^{m-1} \gamma(\ovl{a^{m-i}},a\inv) = \udb{\gamma({a^{m-0}},a\inv)}_{=x_m\inv} \udb{\prod_{i=1}^{m-1}\gamma({a^{m-i}},a\inv)}_{=1} = x_m\inv.$$

We claim by induction on $t$ that $\prod_{i=m}^{m+n-1}(\gamma(\overline{a^i},a)) = x_m^t.$ The base case $t=0$ gives 
$$\prod_{i=m}^{2m-2}(\gamma(\overline{a^i},a)) = \prod_{i=0}^{m-2}(\gamma(\overline{a^i},a)) =1.$$
Assuming the hypothesis, 
\begin{align*}
\prod_{i=m}^{m(t+2)-2}(\gamma(\overline{a^i},a)) &= \left(\prod_{i=m}^{m(t+1)-2}(\gamma(\overline{a^i},a))\right) \left(\prod_{i=m(t+1)-1}^{m(t+2)-2}(\gamma(\overline{a^i},a))\right)\\
&= x_m^{t-1}\left(\prod_{i=m-1}^{2m-2}(\gamma(\overline{a^i},a))\right) \\
&=x_m^{t-1} \cdot \gamma(a^{m-1},a)\left(\prod_{i=m}^{2m-2}(\gamma(\overline{a^i},a))\right) \\
&= x_m^{t-1} \cdot x_m \udb{\left(\prod_{i=0}^{m-2}(\gamma(\overline{a^i},a))\right)}_{=1}.
\end{align*}
Therefore $\tau(a^{m-1}ba^nba^{-m}) = \gamma(a^{m-1},b)x_m^t \gamma(\ovl{m+n},b) x_m\inv = x_{m-1}x_m^t x_{m-1} x_m\inv$. This finishes the proof for the presentation of $H$.
\end{proof}

\section{No Acylindrically Hyperbolic Group Admits a Quasi-geodesic Regular Positive Cone}\label{acyl}

The goal of this section is to prove Theorem \ref{main-thm}. We will do so by showing that for every acylindrically hyperbolic group $G$ (see \cite{Osin2013} for an in-depth discussion) and for every quasi-geodesic language $L$ representing a subset $P$ of $G$, there is a subgroup isomorphic to the free group on two elements which is language-convex with respect to $L$ in $G$. If $P$ is a positive cone, this creates a contradiction to Theorem \ref{herm-sunic}. 

\subsection{Acylindrically hyperbolic groups}\label{acyl-hyp-grp}
A group $G$ has an \emph{acylindrical action} on a metric space $\mathcal{X}$ with distance $d$ if for all $\epsilon > 0$, there exists non-negative constants $R$ and $N$, both depending on $\epsilon$, such that for every two points $a,b \in \mathcal{X}$ satisfying $d(a,b) \geq R$, the set of elements $g \in G$ satisfying
$$d(a,ga) \leq \epsilon \quad \text{ and } \quad d(b,gb) \leq \epsilon$$ is at most $N$.
If a group admits a non-elementary, acylindrical, isometric action on a $\delta$-hyperbolic space (a space where all geodesic triangles have the property that the union of any two sides is contained in a $\delta$-neighbourhood of the third side), then the group is called $\emph{acylindrically hyperbolic}$. Some examples include the mapping class groups of closed, oriented surfaces, groups of outer automorphism of free groups, and free products.

\subsection{Quasi-geodesic positive cones}

Let $\lambda$ and $\epsilon$ be real constants such that $\lambda \geq 1$ and $\epsilon \geq 0$. Let $I$ be an interval of the real line. Let $\mathcal{X}$ be a metric space with metric $d$. A \emph{$(\lambda, \epsilon)$-quasi-geodesic} is a map $\gamma: I \to \mathcal{X}$ such that for all $a,b \in I$, we have the inequalities

$$\frac{|a - b|}{\lambda} - \epsilon \leq d(\gamma(a), \gamma(b)) \leq \lambda |a - b| + \epsilon.$$

Recall from Section \ref{word-ind-path} that words induce paths in the corresponding Cayley graph $\Gamma$. These paths may be viewed naturally as continuous maps. A $(\lambda, \epsilon)$-\emph{quasi-geodesic word} $w$ is a word which induces a $(\lambda,\epsilon)$-quasi-geodesic path $\gamma: [0,n] \to \Gamma$ where $n$ is the length of $w$. 

\begin{definition}[Quasi-geodesic positive cone language]\label{def-qgpc} Let $G = \langle X \rangle$, where $X$ is finite and closed under formal inversion, $X = X\inv$. Let $P$ be any positive cone for $G$. We say that $L$ is a \emph{quasi-geodesic positive cone language} if $L \subset X^*$ and $L$ satisfies the following two conditions.

\begin{enumerate} 
	\item Under the evaluation map $\pi: X^* \to G$ we have that $\pi(L) = P$.
	\item There exists some constants $\lambda$ and $\epsilon$ with $\lambda \geq 1$ and $\epsilon \geq 0$ for which every word $w \in L$ is a $(\lambda,\epsilon)$-quasi-geodesic word. 
\end{enumerate}
\end{definition}

Theorem \ref{main-thm} says that if $G$ is a finitely generated, acylindrically hyperbolic group which admits a quasi-geodesic positive cone language $L$, then $L$ cannot be accepted by any finite state automaton. The following lemma concerns the existence of a hyperbolically embedded subgroup (see \cite[Section 2.1]{Dahmani2011} for a rather long definition). It is not strictly necessary to know the definition of hyperbolically embedded to follow the next results.


\begin{lemma}\label{F_2-embed} If $G$ is an acylindrically hyperbolic group, then there exists a hyperbolically embedded subgroup $H$ of $G$ that is isomorphic to $F_2$, the free group of two elements.
\end{lemma}
\begin{proof}

Osin proved in \cite[Theorem 1.2]{Osin2013} that $G$ being acylindrically hyperbolic is equivalent to containing a proper infinite hyperbolically embedded subgroup. All we need for this proof is the result of Dahmani, Guirardel and Osin in \cite[Section 6.2] {Dahmani2011} which is dependent on the existence of a proper infinite hyperbolically embedded subgroup in $G$. The result states that if $G$ contains a proper infinite hyperbolically embedded subgroup, then for any $n \in \nat$ there exists a subgroup $H \leq G$ such that $H$ is hyperbolically embedded in $G$ and $H \cong F_n \times N$, where $F_n$ is a free group of rank $n$ and $N$ is the maximal finite normal subgroup of $G$. In particular, there exists a hyperbolically embedded subgroup $H \leq G$ such that $H \cong F_2$.
\end{proof}
\begin{lemma}\label{H-L} If $H$ is a hyperbolically embedded subgroup of an acylindrically hyperbolic group $G$, then $H$ is language-convex with respect to every quasi-geodesic language $L$.
\end{lemma} 

\begin{definition}[Morse property]
A subspace $\mathcal{Y}$ of a metric space $\mathcal{X}$ is said to be \emph{Morse} if for every $\lambda \geq 1$ and $\epsilon \geq 0$,  there exists a non-negative constant $R$ depending on $\lambda$ and $\epsilon$ with the property that all $(\lambda,\epsilon)$-quasi-geodesics in $\mathcal{X}$ whose endpoints are in $\mathcal{Y}$ are contained in the neighbourhood of radius $R$ around $\mathcal{Y}$. 
\end{definition}

\begin{proof}[Proof of lemma \ref{H-L}]
Our lemma is largely a consequence of Sisto's theorem in \cite[Theorem 2]{Sisto2013}, which says the following. Let $G$ be a finitely generated group and let $H$ be a finitely generated subgroup that is hyperbolically embedded. Let $\Gamma$ be the Cayley graph of $G$ with respect to the finite generating set $X$ such that $L \subseteq X^*$. The embedding of $H$ in $\Gamma$ has the Morse property.

Thus, there exists an $R = R(\lambda, \epsilon)$ such that for every $(\lambda, \epsilon)$-quasi-geodesic word $u$ with the property that $\bar u \in H$, the induced path lies within $R$ of the embedding of $H$. In particular, this shows that $H$ is language-convex with respect to $L$.
\end{proof}

\begin{corollary}\label{acyl-contra} Let $G$ be a finitely generated acylindrically hyperbolic group with positive cone $P$. If there exists a regular quasi-geodesic positive cone language $L$ representing $P$, then there exists a regular positive cone language for $F_2$. 
\end{corollary}
\begin{proof}
By Lemma \ref{F_2-embed}, we may assume there exists a hyperbolically embedded subgroup $H$ which is isomorphic to $F_2$. The subgroup $H$ is language-convex with respect to $L$ by Lemma \ref{H-L}, which means by Proposition \ref{convex-crit} that $H \cap P$ is a regular positive cone for $H \cong F_2$.
\end{proof} 

However, Hermiller and \v{S}uni\'c's theorem (Theorem \ref{herm-sunic}) states that there is no regular language representing a positive cone of $F_2$, contradicting the assumption of Corollary \ref{acyl-contra}. This proves Theorem \ref{main-thm}.

\end{document}